\documentclass{article}

\usepackage{amssymb, amsmath, amsthm}

\usepackage[all]{xypic}

\newtheorem{theorem}{Theorem}[section]
\newtheorem{lemma}[theorem]{Lemma}
\newtheorem{corollary}[theorem]{Corollary}
\newtheorem{proposition}[theorem]{Proposition}
\newtheorem{conjecture}[theorem]{Conjecture}

\theoremstyle{definition}
\newtheorem{definition}[theorem]{Definition}

\theoremstyle{remark}
\newtheorem{remark}[theorem]{Remark}

\numberwithin{equation}{section}

\DeclareMathOperator{\Kind}{\textit{K}-index}
\DeclareMathOperator{\Gind}{\textit{G}-index}

\DeclareMathOperator{\Ad}{Ad}

\DeclareMathOperator{\Spin}{Spin}
\DeclareMathOperator{\SO}{SO}

\DeclareMathOperator{\Res}{Res}
\DeclareMathOperator{\DInd}{D-Ind}
\DeclareMathOperator{\KInd}{K-Ind}
\DeclareMathOperator{\pHInd}{pH-Ind}

\DeclareMathOperator{\pHPS}{pHamPS}
\DeclareMathOperator{\CpHPS}{CpHamPS}
\DeclareMathOperator{\CHPS}{CHamPS}
\DeclareMathOperator{\ECpHPS}{ECpHamPS}

\DeclareMathOperator{\MWR}{MWR}

\DeclareMathOperator{\se}{se}

\begin{document}

\newcommand{\D}{\slash \!\!\!\! D}

\newcommand{\g}{\mathfrak{g}}
\newcommand{\h}{\mathfrak{h}}
\newcommand{\kk}{\mathfrak{k}}
\newcommand{\m}{\mathfrak{m}}
\newcommand{\p}{\mathfrak{p}}
\newcommand{\aaa}{\mathfrak{a}}
\newcommand{\n}{\mathfrak{n}}
\newcommand{\gse}{\mathfrak{g}^*_{\mathrm{se}}}
\newcommand{\kse}{\mathfrak{k}^*_{\mathrm{se}}}
\newcommand{\Ldom}{\Lambda^*_+}

\newcommand{\Gps}{\hat G_{\text{ps}}}

\newcommand{\clambda}{c(\lambda)}
\newcommand{\pilamrho}{[\pi^p_{\lambda+\rho}]}
\newcommand{\pidmu}{[\pi^d_{\mu}]}
\newcommand{\Rdmu}{R_G^{d, \mu}}

\newcommand{\C}{\mathbb{C}}
\newcommand{\R}{\mathbb{R}}
\newcommand{\Z}{\mathbb{Z}}
\newcommand{\N}{\mathbb{N}}

\newcommand{\B}{\mathcal{B}}
\newcommand{\K}{\mathcal{K}}
\renewcommand{\O}{\mathcal{O}}
\newcommand{\E}{\mathcal{E}}
\newcommand{\SSS}{\mathcal{S}}
\newcommand{\X}{\mathfrak{X}}
\newcommand{\HH}{\mathcal{H}}

\newcommand{\KK}{K \! K}
\newcommand{\Bigwedge}{\textstyle{\bigwedge}}

\newcommand{\QSpin}{Q_{\Spin}}

\newcommand{\ddt}{\left. \frac{d}{dt}\right|_{t=0}}

\title
{Quantisation of presymplectic manifolds, $K$-theory and group representations}


\author{Peter Hochs}



\date{}



\maketitle

\begin{abstract}
Let $G$ be a semisimple Lie group with finite component group, and let $K<G$ be a maximal compact subgroup. We obtain a quantisation commutes with reduction result for actions by $G$ on manifolds of the form $M = G\times_K N$, where $N$ is a compact prequantisable Hamiltonian $K$-manifold. The symplectic form on $N$ induces a closed two-form on $M$, which may be degenerate. We therefore work with presymplectic manifolds, where we take a presymplectic form to be a closed two-form. 
For complex semisimple groups and semisimple groups with discrete series, the main result reduces to 
 results with a more direct representation theoretic interpretation. The result for the discrete series is a generalised version of an earlier result by the author. In addition,  the generators of the $K$-theory of the $C^*$-algebra of a semisimple group are realised as quantisations of fibre bundles over suitable coadjoint orbits. 
\end{abstract}

\tableofcontents


\section{Introduction}

\subsection{Quantisation and reduction}

Let  $G$ be a Lie group acting on a symplectic manifold $(M, \omega)$ in Hamiltonian fashion. Let $\pi$ be an irreducible representation of $G$ associated to (the coadjoint orbit through) an element $\xi \in \g^*$, the dual of the Lie algebra of $G$. Then 
the \emph{quantisation commutes with reduction} principle states that
\[
R^{\pi}_G \bigl( Q_G(M, \omega) \bigr)= Q(M_{\xi}, \omega_{\xi}),
\]
where $Q_G$ and $Q$ denote geometric quantisation, and
 the quantum reduction map $R^{\pi}_G$ is defined by taking multiplicities of $\pi$. Furthermore, $(M_{\xi}, \omega_{\xi})$ is the Marsden--Weinstein reduction \cite{MW} of $(M, \omega)$ at $\xi$, i.e.\ $M_{\xi} = \Phi_M^{-1}(\xi)/G_{\xi}$, with $\Phi_M: M \to \g^*$ a momentum map for the action. If $M$ and $G$ are compact, this principle has been given explicit meaning, and been proved, by Meinrenken and others \cite{Meinrenken, MS, Paradan1, TZ}. The geometric quantisation of $(M, \omega)$ is then defined as the equivariant index of a Dirac operator $ \D_M^L$ on $M$, coupled to a line bundle $L \to M$ with Chern class $[\omega]$:
\begin{equation} \label{eq def quant cpt}
Q_G(M, \omega) = \Gind \left( \D_M^L\right).
\end{equation}

For $M$ and/or $G$ noncompact, results have been achieved in two directions. For compact $G$ and noncompact $M$, there are results by  Ma and Zhang \cite{MaZ, MaZ2} and by Paradan \cite{Paradan2, Paradan3, Paradan4}. If $M$ and $G$ are both allowed to be noncompact, but the orbit space $M/G$ is still compact, Landsman \cite{Landsman} has proposed a definition based on the analytic assembly map $\mu_M^G$ used in the Baum--Connes conjecture \cite{BCH}:
\begin{equation} \label{eq def quant}
Q_G(M, \omega) = \mu_M^G\left[\D_M^L\right] \in K_0(C^*_{(r)}(G)).
\end{equation}
Here $\left[\D_M^L\right] $ denotes the $K$-homology class  \cite{HR} of the Dirac operator $\D_M^L$ on $M$, and $K_0(C^*_{(r)}(G))$ is the $K$-theory group of the (full or reduced) group $C^*$-algebra of $G$. This definition reduces to \eqref{eq def quant cpt} in the compact case. Results based on this definition have been achieved by Landsman and the author  \cite{HL, HochsDS}, and by Mathai and Zhang \cite{MZ}.

The paper \cite{HochsDS} contains a result about reduction at discrete series representations of real semisimple Lie groups, in terms of Landsman's definition of quantisation. This result was based on a quantisation commutes with  \emph{induction} principle, which allowed us to deduce the result from the compact case.



\subsection{Presymplectic manifolds}

The key assumption in \cite{HochsDS} that made the Hamiltonian induction construction used there possible, is that the momentum map $\Phi_M: M \to \g^*$ of the action in question takes values in the strongly elliptic set $\g^*_{\text{se}} \subset \g^*$ of elements with compact stabilisers under the coadjoint action. This was shown to imply that 
\begin{equation} \label{eq M fibre}
M = G\times_K N,
\end{equation}
for some Hamiltonian $K$-manifolds $N$. 
However, many Lie groups, such as complex semisimple ones, have empty strongly elliptic sets. 
Therefore, we only assume that $M$ has the form \eqref{eq M fibre} in this paper, so we are able to treat more general situations than the one considered in \cite{HochsDS}. Because of this weaker assumption,  we will lose some desirable properties of Hamiltonian induction. The most important of these is nondegeneracy of induced symplectic forms. 
We will therefore consider quantisation of \emph{presymplectic manifolds}.

If one leaves the setting of symplectic manifolds and Hamiltonian group actions, a natural question is how to define presymplectic manifolds precisely, i.e.\ what structure should be retained to define meaningful notions of quantisation and reduction. In \cite{CdSKT}, a hierarchy of structures that can be used to define geometric quantisation is discussed: complex structures, stable complex structures and $\Spin^c$-structures. Of these, $\Spin^c$-structures are the most generally applicable. The quantisation commutes with induction principle in \cite{HochsDS} includes induction of $\Spin^c$-structures, so that the induced manifolds can indeed be quantised.

While a $\Spin^c$-structure is enough to define a Dirac operator, and hence geometric quantisation, other structures are needed to define reduction. In \cite{Bottacin, EEMLRR}, it is shown that the notions of momentum map and symplectic reduction generalise directly to the presymplectic setting. In \cite{EEMLRR}, a presymplectic form is defined as a closed two-form with constant rank, whereas in \cite{Bottacin}, the constant rank assumption is dropped. The latter definition of a presymplectic form, simply as a closed two-form, is also used for example in \cite{CdSKT,  GNH, GK1, GK2, KT}. This definition is compatible with the Hamiltonian induction procedure we use (called \emph{pre-}Hamiltonian induction in this generalised setting), and  for that reason, we will define a presymplectic manifold as a  smooth manifold equipped with a closed two-form.  Since the presymplectic manifolds we consider may be odd-dimensional, the quantisation of these spaces may end up in \emph{odd} $K$-theory. Then the even $K$-theory group $K_0(C^*_{(r)}(G))$ in \eqref{eq def quant} should be replaced by $K_1(C^*_{(r)}(G))$.

Previous results on presymplectic manifolds and their quantisations have been obtained by Cannas da Silva, Grossberg, Karshon and Tolman \cite{CdSKT, GK2, KT}, while Grossberg and Karshon \cite{GK1} have given applications to representation theory. Gotay, Nester and Hind \cite{GNH} have applied the theory to constrained classical mechanics (they discuss Maxwell theory as an example), noting that presymplectic manifolds arise naturally as submanifolds of symplectic manifolds, defined by physical contraints.


\subsection*{Acknowledgements}

The author would like to thank Mathai Varghese and the referee, for useful discussions and comments. 

This research was supported by the Alexander von Humboldt foundation,  through a Marie Curie fellowship of the European Union, and by the Australian Research Council, through Discovery Project DP110100072.


\subsection*{Notation}
We will write $d_X$ for the dimension of a manifold  $X$. In particular, we will write $d := d_{G/K}$. Where appropriate, these dimensions should be interpreted modulo 2.


\section{Results}

We state a quantisation commutes with reduction result, and specialise this result to reduction at connected components of the principal series of complex semisimple groups, and at discrete series representations of real semisimple groups. As a special case, we show how a canonical generator of the $K$-theory of the reduced $C^*$-algebra of a semisimple group can be obtained as the quantisation  of a \emph{fibre bundle} over the associated coadjoint orbit. This coadjoint orbit is a symplectic manifold, but the pullback of the symplectic form to the total space of this fibre bundle is only presymplectic.

\subsection{The main result} \label{sec results}


Let $G$ be an almost connected semisimple Lie group, with a maximal compact subgroup $K<G$. Let $\g = \kk \oplus \p$ be a Cartan decomposition at the Lie algebra level, and write $d:= \dim G/K = \dim \p$. Chabert, Echterhoff and Nest proved the Connes--Kasparov conjecture for all almost connected Lie groups \cite{CEN}. It states that
the \emph{Dirac induction} map
\[
\DInd_K^G: R(K) \to K_d(C^*_r(G))
\]
is an isomorphism from the representation ring $R(K)$ of $K$ onto the $K$-theory group $K_d(C^*_r(G))$. (The other $K$-theory group is zero.) See Subsection \ref{sec dirac ind} for more information. 
 Let $\Ldom$ be the set of dominant weights for $K$, with respect to some maximal torus and a choice of positive roots. Let $\rho_c$ be half the sum of these positive compact roots. 
For $\lambda \in \Ldom$, we write
\[
\clambda := \DInd_K^G[V_{\lambda}] \quad \in K_d(C^*_r(G)),
\]
with $V_{\lambda}$ the irreducible representation of $K$ with highest weight $\lambda$. Then $K_d(C^*_r(G))$ is the free abelian group generated by these classes.

 Let $(N, \nu)$ be a compact, $\Spin^c$-prequantisable Hamiltonian $K$-manifold, with momentum map $\Phi_N: N \to \kk^*$. Form the manifold $M := G\times_K N$ as the quotient of $G \times N$ by the free $K$-action given by $k\cdot(g, n) = (gk^{-1}, kn)$, for $k \in K$, $g \in G$ and $n \in N$. Consider the $G$-invariant two-form $\omega$ on $M$ defined as follows. For $n \in N$, $v, w \in T_nN$ and $X, Y  \in \p$, we have the tangent vectors $Tq(X + v)$ and $Tq(Y+w)$ to $M$ at $[e, n]$, where $q: G \times N \to M$ is the quotient map. Then $\omega$ is defined by the properties that it is $G$-invariant, and
\[
\omega_{[e, n]}\bigl(Tq(X+v), Tq(Y+w)\bigr) := \nu_n(v, w) - \langle\Phi_N(n), [X,Y]\rangle.
\]

Then we have the following decomposition of the quantisation of $(M, \omega)$.
\begin{theorem}[Quantisation commutes with reduction for semisimple groups] \label{thm [Q,R]=0}
The pair $(M, \omega)$ is an equivariantly $\Spin^c$-prequantisable presymplectic manifold, and the action of $G$ on $M$ is pre-Hamiltonian. Suppose the image of the momentum map $\Phi_N$ has nonempty intersection with the interior of the positive Weyl chamber.\footnote{If this condition is not satisfied, the is a more subtle $\rho$-shift involved in the statement of the theorem, as discussed in the introduction to \cite{Paradan}.} The $\Spin^c$-quantisation of this action decomposes as
\[
\boxed{
Q_G(M, \omega) = \sum_{\lambda \in \Ldom} Q(M_{\lambda + \rho_c}, \omega_{\lambda + \rho_c}) \clambda,
}
\]
where $(M_{\lambda + \rho_c}, \omega_{\lambda + \rho_c})$ denotes the presymplectic reduction of $(M, \omega)$ by $G$ at $(\lambda + \rho_c)/i$.
\end{theorem}
We will see in Lemma \ref{lem sympl reds} that the presymplectic reductions of $(M, \omega)$ are in fact symplectic manifolds.

Since the Connes--Kasparov conjecture holds for any almost connected Lie group, Theorem \ref{thm [Q,R]=0} may be formulated for any such group. In addition, one may allow manifolds $M$ of more general forms. This leads to the following conjecture.
\begin{conjecture}[Quantisation commutes with reduction]
Theorem \ref{thm [Q,R]=0} is true for any almost connected Lie group $G$, and any equivariantly prequantisable cocompact \mbox{(pre-)} Hamiltonian $G$-manifold $(M, \omega)$ for which $\dim M = \dim(G/K)$ modulo 2.
\end{conjecture}


\subsection{Special cases} \label{sec special cases}

Suppose $G$ is complex semisimple. Let $\rho$ be half the sum of a choice of positive roots of $\g$ with respect to the standard Cartan subalgebra, compatible with the choice of positive compact roots made earlier. We will see
 in Subsection \ref{sec Ktheor CG} that then
 \[
 \clambda = \pilamrho,
 \]
where $\pilamrho \in K_d(C^*_r(G))$ is the class associated to the connected component of the principal series of $G$ corresponding to the parameter $\lambda + \rho$. Hence Theorem \ref{thm [Q,R]=0} reduces to the following statement.
\begin{corollary}[Quantisation commutes with reduction at families of principal series representations] \label{cor [Q,R]=0 ps}
In the setting of Theorem \ref{thm [Q,R]=0}, with $G$ complex semisimple, one has
\[
\boxed{
Q_G(M, \omega) = \sum_{\lambda \in \Ldom} Q(M_{\lambda + \rho_c}, \omega_{\lambda + \rho_c}) \pilamrho.
}
\]
\end{corollary}


Next, suppose $G$ is a semisimple Lie group with discrete series. Then Theorem \ref{thm [Q,R]=0} implies a generalised version of the main result in \cite{HochsDS}. Indeed, let $\pi^d_{\mu}$ be an irreducible discrete series representation of $G$, where  $\mu$ is the Harish--Chandra parameter of $\pi^d_{\mu}$ such that $(\alpha, \mu) \geq 0$ for all positive compact roots $\alpha$ (for some Weyl group invariant inner product). 
Let $\pidmu \in K_0(C^*_r(G))$ be the associated generator, as in \cite{HochsDS, Lafforgue}.
\begin{corollary}[Quantisation commutes with reduction at discrete series representations] \label{cor [Q,R]=0 ds}
The multiplicity of $\pidmu$ in $Q_G(M, \omega)$ equals $(-1)^{d/2} Q(M_{\mu})$.
\end{corollary}

\begin{remark}
Let us show that Corollary  \ref{cor [Q,R]=0 ds} implies
Theorem 1.9 in \cite{HochsDS}, under a slightly stronger condition. Let $\mathfrak{t}^*_+$ be the positive Weyl chamber for a maximal torus $\mathfrak{t} \subset \kk$, corresponding to the choice of positive roots. Let $\mathfrak{t}^*_{+, \se}$ be the chamber $\mathfrak{t}^*_+$, with the Weyl chamber walls associated to noncompact roots removed. In \cite{HochsDS}, it was assumed that
\[
\Phi(M) \subset \gse = \Ad^*(G) \mathfrak{t}^*_{+, \se},
\]
with $\Phi$ the momentum map of the action in question. By Theorem 4.2 in \cite{HochsDS}, this implies that $M$ is of the form 
$M = G\times_K N$ as above.

In the setting of Corollary \ref{cor [Q,R]=0 ds}, let
\[
\Rdmu: K_0(C^*_r(G)) \to \Z
\]
be the reduction map used in \cite{HochsDS}, i.e.\ the map induced by the homomorphism  $\pi^d_{\mu}$ from $C^*_r(G)$
to the compact operators on the representation space of $\pi^d_{\mu}$. Theorem 1.9 in \cite{HochsDS} states that
\begin{equation} \label{eq quant red ds 2}
\Rdmu\bigl(Q_G(M, \omega)\bigr) = (-1)^{d/2} Q(M_{\mu}).
\end{equation}

Now let $\mathfrak{t}^*_{++}$ be the interior of $ \mathfrak{t}^*_{+}$, and suppose that
\[
\Phi(M) \subset \Ad^*(G) \bigl( \mathfrak{t}^*_{++}).
\]
Since $\Lambda^*_+ \cap \mathfrak{t}^*_{++} = \Lambda^*_+ \cap i\gse + \rho_c$,  the terms in the sum in Theorem \ref{thm [Q,R]=0} are only nonzero if $\lambda \in i\gse$. Such a $\lambda$
is of the form $\mu - \rho_c$, with $\mu$ regular. Then $\mu$ is the Harish--Chandra parameter of an irreducible discrete series representation $\pi^d_{\mu}$ (see e.g.\ Theorem 9.20 in \cite{Knapp}), and we will see in \eqref{eq class ds}, that $\clambda = (-1)^{d/2}\pidmu$. (In particular, $\pidmu$ is indeed a generator of $K_0(C^*_r(G))$.) Hence $Q_G(M, \omega)$ decomposes in terms of such generators $\pidmu$.
Since $\Rdmu$ maps the discrete series generator $[\pi^d_{\mu}]$ to one, and all others to zero (see \cite{Lafforgue}, bottom of page 807), Corollary \ref{cor [Q,R]=0 ds} indeed  implies the relation \eqref{eq quant red ds 2}.
\end{remark}


\subsection{Orbit methods} \label{sec orbit results}

As a special case\footnote{This special case is not quite a corollary to Theorem \ref{thm [Q,R]=0}, but obtained in an analogous way. See Subsection \ref{sec orbit}.} of Theorem \ref{thm [Q,R]=0}, we obtain an `orbit method' for the generators $\clambda$ of $K_d(C^*_r(G))$: a realisation of  these classes as quantisations of fibre bundles over coadjoint orbits. If $G$ is complex semisimple, a class $\clambda = \pilamrho$ corresponds, in a sense, to a \emph{bundle} of irreducible representations, so that it seems reasonable to obtain it as the quantisation of a bundle over a coadjoint orbit. If $G$ is a real semisimple Lie groups with discrete series, we realise the generator of $K_0(C^*_r(G))$ associated to a discrete series representation as the quantisation of a coadjoint orbit.

Explicitly, let $\lambda \in \Ldom$ be given, and let $\xi:= (\lambda+\rho_c)/i$. Consider the homogeneous space
$
M^{\lambda} := G/K_{\xi}, 
$
and the coadjoint orbit
$
\O^{\lambda} := G\cdot \xi 
$
of $G$ through $\xi$. Let $\omega^{\lambda}$ be the standard Kostant--Kirillov symplectic form on $\O^{\lambda}$. Let \begin{equation} \label{eq pi}
p: M^{\lambda} \to G/G_{\xi} \cong \O^{\lambda} 
\end{equation}
be the natural projection map, and consider the two-form
$p^*\omega^{\lambda}$ on $M^{\lambda}$. Note that this form is closed since $\omega^{\lambda}$ is, so that $(M^{\lambda}, p^*\omega^{\lambda})$ is a presymplectic manifold.
\begin{proposition}[Orbit method for generators of $K$-theory of group $C^*$-algebras] \label{prop orbit}
The action of $G$ on the presymplectic manifold $(M^{\lambda}, p^*\omega^{\lambda})$  is equivariantly $\Spin^c$-prequantisable. The quantisation of this action is the class
\[
\boxed{
Q_G\bigl(M^{\lambda}, p^*\omega^{\lambda}\bigr) = \clambda \quad \in K_d(C^*_r(G)).
}
\]
\end{proposition}

If $G$ is complex semisimple, the fact that $\clambda = \pilamrho$ means that Proposition \ref{prop orbit} takes the following form.
\begin{corollary} [Orbit method for families of principal series representations] \label{cor orbit ps}
If $G$ is complex semisimple, then
\[
\boxed{
Q_G\bigl(M^{\lambda}, p^*\omega^{\lambda}\bigr) = \pilamrho \quad \in K_d(C^*_r(G)).
}
\]
\end{corollary}

Note that the map $p$ defines a fibre bundle, with contractible fibre $G_{\xi}/K_{\xi}$. If $\lambda$ is a \emph{regular} weight, then we have $G_{\xi} = MA$ and $K_{\xi} = M$ (for a Borel subgroup\footnote{In this sentence only, $M$ denotes a subgroup of $K$, rather than a manifold to be quantised.} $MAN < G$), so that  this fibre equals the subgroup $A$. Then this fibre is homeomorphic to $\aaa$, which in turn is homeomorphic to the connected component of the principal series of $G$ associated to the class $\pilamrho$. 

In the case considered in \cite{HochsDS}, one has $G_{\xi} = K_{\xi}$, so that $M^{\lambda}$ equals the coadjoint orbit $\O^{\lambda}$. Then Proposition \ref{prop orbit} has the following consequence.
\begin{corollary}[Orbit method for the discrete series] \label{cor orbit ds}
Let $G$ be a semisimple Lie group with discrete series. In the notation of Corollary \ref{cor [Q,R]=0 ds}, one has
\[
\boxed{Q_G(\O^{\mu - \rho_c}, \omega^{\mu - \rho_c}) = (-1)^{d/2} \pidmu \quad \in K_0(C^*_r(G)).}
\]
\end{corollary}
(Note that $\O^{\mu - \rho_c}$ is the coadjoint orbit through $\mu/i$.)


\subsection*{Generalisations}

It would be useful and interesting to generalise the results in this paper to cases where the orbit spaces are not necessarily compact (see \cite{MaZ, MaZ2, Paradan3, Paradan4} for such results for compact groups). Recent results in this direction are to appear in \cite{HM}.

\section{Dirac induction and reduction}

The generators $\clambda$ of $K_d(C^*_r(G))$ are defined via \emph{Dirac induction}, which we discuss in Subsection \ref{sec dirac ind}. In Subsection \ref{sec Ktheor CG}, we briefly discuss Penington and Plymen's explicit description of Dirac induction for complex semisimple Lie groups, in terms of  the principal series of such groups.

\subsection{Dirac induction} \label{sec dirac ind}

Dirac induction is a map
\[
\DInd_K^G: R(K) \to K_*(C^*_r(G)),
\]
which is defined in terms of Dirac operators $\D^V$ on $G/K$, coupled to given irreducible representations $V$ of $K$. For this map to be well-defined, we will assume that the representation $\Ad: K \to \SO(\p)$  lifts to $\widetilde{\Ad}: K \to \Spin(\p)$. It may be necessary\footnote{Although, in the complex semisimple case,  Penington and Plymen describe how to handle the case where a lift $\widetilde{\Ad}$ does not exist in Section 6 of \cite{PP}.} to replace $G$ and $K$ by double covers for this lift to exist. 

Let $\Delta_{\p}$ be the standard representation of $\Spin(\p)$. Dirac induction is defined by
\begin{equation}\label{IGK}
\DInd_K^G[V] := \left[ \bigl(C^*_r(G) \otimes \Delta_{\p} \otimes V \bigr)^K, b \bigl({\D}^V \bigr) \right] \quad
\in \KK_{*}(\C, C^*_r(G)) \cong K_{*}(C^*_r(G)),
\end{equation}
where, for an orthonormal basis $\{X_j\}$ of $\p$,
\begin{equation} \label{DV}
{\D}^V := \sum_{j=1}^{d} X_j \otimes c(X_j) \otimes 1_V
\end{equation}
is the $\Spin$-Dirac operator on $G/K$, and $b: \R \to \R$ is a normalising function, e.g.\ $b(x) = \frac{x}{\sqrt{1+x^2}}$. Here in the first factor, $X_j$ is viewed as a vector field on $G$, and in the second factor, $c$ denotes the Clifford action.

If $\p$ is even-dimensional, then $\Delta_{\p}$ splits into two irreducibles: $\Delta_{\p} = \Delta_{\p}^+ \oplus \Delta_{\p}^-$, and the Dirac operator \eqref{DV} is odd with respect to the induced grading on $\E_V$. Therefore, Dirac induction takes values in $\KK_{d}(\C, C^*_r(G))$.

The Connes--Kasparov conjecture states that  that the Dirac induction map is an isomorphism of abelian groups. This was proved for complex semisimple Lie groups $G$ by Penington and Plymen in \cite{PP}, for connected linear reductive groups
by Wassermann \cite{Wassermann}, and for general almost connected Lie groups by Chabert, Echterhoff and Nest \cite{CEN}.

The reduction map for $K$ at a dominant weight $\lambda \in \Ldom$,
\[
R_K^{\lambda}: R(K) \to \Z,
\]
is defined by taking the multiplicity of the irreducible $K$-representation $V_{\lambda}$ with highest weight $\lambda$.
A restatement of the Connes-Kasparov conjecture is that $K_d(C^*_r(G))$ is the free abelian group with generators
\[
\{\clambda := \DInd_K^G[V_{\lambda}]; \lambda \in \Ldom\}.
\]
For $\lambda \in \Ldom$, let
\begin{equation} \label{eq RG}
R_G^{\lambda}: K_d(C^*_r(G)) \to \Z
\end{equation}
be the reduction map defined by
\[
R_G^{\lambda}\biggl(\sum_{\lambda' \in \Ldom} m_{\lambda'} c(\lambda') \biggr) = m_{\lambda}.
\]
By definition of this map and the generators $\clambda$, the following diagram commutes:
\begin{equation} \label{eq RG RK}
\xymatrix{
K_*(C^*_r(G)) \ar[dr]^-{R_G^{\lambda}} & \\
R(K) \ar[u]^{\DInd_K^G} \ar[r]_-{R_K^{\lambda}} & \Z.
}
\end{equation}

\subsection{Complex semisimple groups} \label{sec Ktheor CG}

In \cite{PP}, Penington and Plymen prove the Connes--Kasparov conjecture for complex semisimple Lie groups, via a very explicit description of the $K$-theory of the $C^*$-algebra of such groups, in terms of their principal series representations. We briefly recall the parts of their work we will use. For details and proofs, we refer to \cite{PP}, and references therein.

Suppose  $G$ is a complex semisimple Lie group. Let $\g = \kk \oplus \aaa \oplus \n$ be an Iwasawa decomposition at the Lie algebra level. The principal series representations of $G$ are parametrised by the set
\begin{equation} \label{eq decomp Gps}
\Gps = \bigcup_{\lambda \in \Ldom} E_{\lambda},
\end{equation}
where $E_{\lambda}$ is is a connected component of $\Gps$, homeomorphic to the quotient of $\aaa$ by a certain Weyl group action. This Weyl group is trivial if and only if $\lambda$ is \emph{regular}, i.e.\ lies inside $\Ldom + \rho$, where $\rho$ is half the sum of a choice of positive roots of $\g$ with respect to the standard Cartan subgroup $Z_{\kk}(\aaa) \oplus \aaa$.

By Proposition 4.1 in \cite{PP}, one has
\[
K_*(C^*_r(G)) = K^*(\Gps),
\]
the topological $K$-theory of $\Gps$. Therefore,
\[
K_d(C^*_r(G)) = \bigoplus_{\lambda \in \Ldom} \Z \cdot \pilamrho,
\]
with $\pilamrho \in K^d(E_{\lambda + \rho})$ the Bott generator, while $K_{d+1}(C^*_r(G)) = 0$.
We will interpret the $K$-theory class $\pilamrho \in K_*(C^*_r(G))$ as the generator associated to the connected component $E_{\lambda + \rho}$ of $\Gps$. 

In Section 5 of \cite{PP}, Penington and Plymen show that for complex semisimple Lie groups, the generators $\clambda$ of $K_d(C^*_r(G))$ equal the classes $\pilamrho$:
\begin{equation} \label{eq lambda pilamrho}
\clambda = \pilamrho.
\end{equation}
Therefore, the reduction map \eqref{eq RG} is now given by
\[
R_G^{\lambda}\biggl(\sum_{\lambda' \in \Ldom} m_{\lambda' + \rho} [\pi^p_{\lambda' + \rho}] \biggr) = m_{\lambda + \rho}.
\]

\section{Pre-Hamiltonian induction and quantisation}

In this section, we consider any connected Lie group $G$ with a maximal compact subgroup $K < G$, and an $\Ad(K)$-invariant subspace $\p \subset \g$ such that $\g = \kk \oplus \p$. We will of course later apply what follows to semisimple groups $G$. We equip $\g$ with any $\Ad(K)$-invariant inner product with respect to which $\kk \perp \p$.
 As in Subsection \ref{sec dirac ind}, we assume that the representation $\Ad: K \to \SO(\p)$ lifts to $\widetilde \Ad: K \to \Spin(\p)$, so that $G/K$ is a $\Spin$ manifold. 

\subsection{Pre-Hamiltonian induction}

In \cite{HochsDS}, the notion of \emph{Hamiltonian induction} was introduced, which assigns a Hamiltonian $G$-manifold $M = G\times_K N$, with an equivariant prequantisation, to a prequantised Hamiltonian $K$-manifold $N$. There is an inverse construction called \emph{Hamiltonian cross-section}. A crucial assumption in \cite{HochsDS} was that the momentum map $\Phi_M: M \to \g^*$ takes values in the strongly elliptic set $\gse$, i.e.\ the set of elements of $\g^*$ with compact stabilisers under the coadjoint action. We will now apply this construction to any semisimple group $G$, for which $\gse$ may be empty. Therefore, we drop the assumption that $\Phi_M(M) \subset \gse$. This has two main consequences for the Hamiltonian induction process:
\begin{itemize}
\item the two-form on $M$ induced by the symplectic form on $N$ may be degenerate;
\item Hamiltonian cross-sections are no longer well-defined (specifically, the subsets $N \subset M$ taken in this process may not be smooth submanifolds).
\end{itemize}
Because of the latter point, we will not be able to use Hamiltonian cross-sections, but need to \emph{assume} that $M$ has the form $G\times_K N$. The first point means we will have to deal with \emph{presymplectic manifolds}. There are different variations of the definition of presymplectic manifolds (sometimes the rank of the two-form considered is assumed to be constant). As in \cite{Bottacin, CdSKT,  GNH, GK1, GK2, KT}, we will simply consider manifolds with closed two-forms on them.
\begin{definition}
A \emph{presymplectic form} on a smooth manifold $M$ is a closed two-form $\omega \in \Omega^2(M)$. The pair $(M, \omega)$ is then called a \emph{presymplectic manifold}.
\end{definition}

The definition of a momentum map for an action by a Lie group on a presymplectic manifold is completely analogous to the symplectic case \cite{Bottacin, EEMLRR}. If such a map exists, we call the action \emph{pre-Hamiltonian}. Presymplectic reduction is then defined analogously to symplectic reduction.
A prequantisation of  a presymplectic manifold $(M, \omega)$ is also defined as in the symplectic case. In particular, we will use $\Spin^c$-prequantisation, which involves a line bundle $L^{2\omega} \to M$ with a connection whose curvature is $2\pi i \cdot 2\omega$, and a $\Spin^c$-structure\footnote{We will slightly abuse terminology, by using the term $\Spin^c$-structure on a manifold $X$ for a principal $\Spin^c$-bundle $P \to X$ such that $P\times_{\Spin^c(d_X)} \R^{d_X} \cong TX$, without explicitly referring to this isomorphism.} $P^M \to M$ with determinant line bundle $L^{2\omega}$.

In the presymplectic setting, we will call the process analogous to Hamiltonian induction \emph{pre-Hamiltonian induction}. It is defined completely analogous to Hamiltonian induction, as described in Sections 2 and 3 of \cite{HochsDS}. We will briefly review the constructions here.

For any Lie group $H$, let $\pHPS(H)$ be the set
of pre-Hamiltonian $H$-actions with equivariant $\Spin^c$-prequantisations and $\Spin^c$-structures, 
which consists of classes of septuples
\[
\bigl(M, \omega, \Phi_M, {L^{2\omega}}, (\relbar, \relbar)_{L^{2\omega}}, \nabla^M, P^M \bigr), 
\]
where
\begin{itemize}
\item $(M, \omega)$ is a presymplectic manifold, equipped with an $H$-action that preserves $\omega$;
\item $\Phi_M: M \to \h^*$ is a momentum map for this action;
\item $\bigl({L^{2\omega}}, (\relbar, \relbar)_{{L^{2\omega}}}, \nabla^M \bigr)$ is an $H$-equivariant $\Spin^c$-prequantisation of $(M, \omega)$;
\item $P^M \to M$ defines an $H$-equivariant $\Spin^c$-structure on $M$, with determinant line bundle $L^{2\omega}$, such that $M$ is complete in the Riemannian metric induced by $P^M$.
\end{itemize}
Two of such septuples are identified if there is an equivariant diffeomorphism between the manifolds in question, which relates the presymplectic forms, momentum maps, line bundles and metrics on them via pullback.\footnote{We do not explicitly require that such a diffeomorphism relates the connections and $\Spin^c$-structures of two such septuples to each other. All that is needed for the purposes of geometric quantisation is that it relates the curvatures of connections, and the determinant line bundles of $\Spin^c$-structures, and this follows from the other properties.}

We will also use the sets 
$\CpHPS(H)$, for which  $M/H$ should be compact, 
$\CHPS(H)$, for which in addition $\omega$ should be symplectic,  and
$\ECpHPS(H) \subset \CpHPS(H)$, for which $M$ should be even-dimensional.

\begin{definition}
\emph{Pre-Hamiltonian induction} is the map\footnote{The author wishes to point out that the occurrence of his initials in the notation for the pre-Hamiltonian induction map is purely coincidental.}
\[
\pHInd_K^G: \pHPS(K) \to \pHPS(G),
\]
given by 
\begin{multline} \label{eq ind}
\pHInd_K^G\bigl[N, \nu, \Phi_N, L^{2\nu}, (\relbar, \relbar)_{L^{2\nu}}, \nabla^N, P^N\bigr] = \\
	\bigl[M, \omega, \Phi_M,  L^{2\omega}, (\relbar, \relbar)_{L^{2\omega}}, \nabla^M, P^M \bigr],
\end{multline}
as defined below.
\begin{itemize}
\item The manifold $M = G \times_{K} N$ is the quotient of $G \times N$ by the $K$-action defined by $k(g, n) = (gk^{-1}, kn)$, for $k \in K$, $g \in G$ and $n \in N$.
\item The $G$-invariant two-form $\omega \in \Omega^2(M)$ is defined by 
\begin{equation} \label{eq ind form}
\omega_{[e, n]}\bigl(Tq(v + X), Tq(w + Y)\bigr) = \nu_n(v, w) - \langle \Phi_N(n), [X,Y]\rangle,
\end{equation}
for $n \in N$, $v, w \in T_nN$ and $X, Y \in \p$, where we note that $T_{[e,n]}M \cong T_nN \oplus \p$ via the tangent map $Tq$ of the quotient map $q: G\times N \to M$.
\item The momentum map $\Phi_M: M \to \g^*$ is defined by 
\[
\Phi_M([g, n]) = \Ad^*(g)\Phi_N(n),
\]
for $g \in G$ and $n \in N$.
\item The line bundle $L^{2\omega}$ equals $G \times_K L^{2\nu} \to M$.
\item The Hermitian metric $(\relbar, \relbar)_{L^{2\omega}}$ on $L^{2\omega}$ is given by
\[
\bigl([g, l], [g', l']\bigr)_{L^{2\omega}} = (l,l')_{L^{2\nu}},
\]
for $g, g' \in G$, $n \in N$ and $l, l' \in L^{2\nu}_N$.
\end{itemize}
The definitions of the connection $\nabla^M$ on $L^{2\omega}$ and the $\Spin^c$-structure $P^M$ on $M$ are more involved, and we refer to Section 3 of \cite{HochsDS} for details.
\end{definition}

We then have the following result, which can be proved completely analogously to Sections 2 and 3 in \cite{HochsDS}, omitting the assumption that momentum maps take values in strongly elliptic sets.
\begin{theorem}[pre-Hamiltonian induction] \label{thm pHInd}
Pre-Hamiltonian induction is well-defined, in the sense that $(M, \omega)$ is a presymplectic manifold, $\Phi_M$ is a momentum map, $\bigl(L^{2\omega}, (\relbar, \relbar)_{L^{2\omega}}, \nabla^M \bigr)$ is a $G$-equivariant $\Spin^c$-prequantisation of $(M, \omega)$, and $P^M \to M$ defines a $G$-equivariant $\Spin^c$-structure on $M$, with determinant line bundle $L^{2\omega}$.
\end{theorem}

\subsection{Quantisation commutes with induction}

As in \cite{HL, HochsDS, Landsman}, we define geometric quantisation as the analytic assembly map used in the Baum--Connes conjecture \cite{BCH} applied to the classes defined by Dirac operators in $K$-homology \cite{HR}.

Consider a class
\[
\bigl[M, \omega, \Phi_M,  L^{2\omega}, (\relbar, \relbar)_{L^{2\omega}}, \nabla^M, P^M \bigr]  \in \CpHPS(G).
\]
Then one has a $\Spin^c$-Dirac operator \cite{Duistermaat, Friedrich, LM}  $\D_M^{L^{2\omega}}$ on  $M$, coupled to the line bundle $L^{2\omega}$ in question via the given connection. This Dirac operator defines a $K$-homology class
$
\left[\D_M^{L^{2\omega}}\right] \in K_{d_M}^G(M).
$
\begin{definition}
The \emph{quantisation map}
\[
Q_G: \CpHPS(G) \to K_{d_M}(C_r^*(G)),
\]
is  defined by 
\[
Q_G\bigl[M, \omega, \Phi_M,  (\relbar, \relbar)_{L^{2\omega}}, \nabla^M, P^M \bigr]  = \mu_M^G\left[\D_M^{L^{2\omega}}\right],
\]
where the map $\mu_M^G$ is the analytic assembly maps
\end{definition}

The quantisation map
\[
Q_K: \CpHPS(K) \to K_{d_N}(C^*_r(K))
\]
is defined analogously. Restricting to even-dimensional manifolds $N$, and noting that $\mu_N^K: K_0^K(N) \to K_0(C^*_r(K))$ then equals the usual equivariant index
\[
\Kind:K_0^K(N) \to R(K),
\]
we have the following result.
\begin{theorem}[Quantisation commutes with pre-Hamiltonian induction] \label{thm quant ind}
The following diagram commutes:
\[
\xymatrix{
\CpHPS(G) \ar[r]^-{Q_G} & K_{d}(C^*_r(G)) \\
\ECpHPS(K) \ar[u]^{\pHInd_K^G} \ar[r]^-{Q_K} & R(K), \ar[u]_{\DInd_K^G}
}
\]
where $\DInd_K^G$ is the Dirac induction map \eqref{IGK}.
\end{theorem}
\begin{proof}
We will discuss  how the proof of Theorem 4.5 in \cite{HochsDS} should be modified to apply to the current setting. One modification
is a correction to the argument in \cite{HochsDS}, and another is the use of  the dimension $d$ of $G/K$. 

Consider the following version of Diagram (28) from \cite{HochsDS}:
\begin{equation}\label{diagram}
\xymatrix{
K_0^{G}(M) \ar[r]^{\mu_M^G} & K_0(C^*_r(G)) \\
K_0^{G \times \Delta(K)}(G \times N) \ar[r]^{\mu_{G \times N}^{G \times \Delta(K)}} \ar[u]_{V_{\Delta(K)}} & K_0(C^*_r(G \times K)) \ar[u]_{R^0_{K}} \\
K_0^{G \times K \times K}(G \times N) \ar[r]^{\mu_{G \times N}^{G \times K \times K}} \ar[u]^{\Res^{G \times K \times K}_{G \times \Delta(K)}}
		& K_0(C^*_r(G \times K \times K)) \ar[u]_{\Res^{G \times K \times K}_{G \times \Delta(K)}} \\
K_0^K(N) \ar[r]^{\mu_N^K} \ar@/^5pc/[uuu]^{\KInd_K^G} \ar[u]_{[{\D}_{G}] \times \relbar}
		& R(K).  \ar[u]_{\mu_{G}^{G \times K}[{\D}_{G}] \times \relbar} 
}
\end{equation}
The difference with Diagram (28) in \cite{HochsDS}, is that the operator $\D_{G, K}$ used in \cite{HochsDS}, which is not elliptic, is replaced by the operator
\[
\D_G := \sum_{j=1}^{d_G} X_j \otimes c(X_j)
\]
on $C^{\infty}(G) \otimes \Delta_{\g}$. Here $\{X_j\}_{j=1}^{d_G}$ is a basis of $\g$, orthonormal with respect to the standard inner product defined via the Killing form and a Cartan involution. Suppose $\{X_1, \ldots, X_d \}$ is a basis of $\p$, so that this notation is compatible with \eqref{DV}. 

Consider the action by $G\times K$ on $C^{\infty}(G) \otimes \Delta_{\g}$ given by
\[
\bigl((g, k)\cdot (f\otimes \delta)\bigr)(g') = f(g^{-1}g'k)\otimes \widetilde{\Ad}(k)\delta,
\]
where $g, g' \in G$, $k\in K$, $f \in C^{\infty}(G)$, $\delta \in \Delta_{\g}$, and $\widetilde{\Ad}$ denotes the composition
\[
K\xrightarrow{\widetilde{\Ad}} \Spin(\p) \hookrightarrow \Spin(\g).
\]
That is, $K$ acts trivially on the component $\Delta_{\kk}$ of $\Delta_{\g} = \Delta_{\kk} \otimes \Delta_{\p}$. Since $\D_G$ is elliptic, essentially self-adjoint,  $G\times K$-equivariant, and graded if $d_G$ is even,  it defines a class
\[
[\D_G] \in K_{d_G}^{G\times K}(G).
\]

Commutativity of the left hand part of Diagram \eqref{diagram} is the definition of the $K$-induction map $\KInd_K^G$.
By  Subsection 5.1, Lemma 5.1 and Theorem 5.2 in \cite{HochsDS}, the rest of Diagram \eqref{diagram} commutes as well.

By Lemma 6.2 in \cite{HochsDS} (with $\D_{G, K}$ replaced by $\D_G$), the map 
$\KInd_K^G$ maps the class $\bigl[\D_N^{L^{2\nu}} \bigr] \in K_0^K(N)$ to the class of the elliptic operator
\[
\bigl(\D_G \otimes 1 + 1\otimes\D_N^{L^{2\nu}}\bigr)^K
\]
on
\begin{equation} \label{eq invar sections}
\bigl(C^{\infty}(G) \otimes \Delta_{\g} \otimes \Gamma^{\infty}(N, \SSS_N \otimes L^{2\nu})\bigr)^K,
\end{equation}
with $\SSS_N$ the spinor bundle on $N$.

The space \eqref{eq invar sections} equals the space of sections of the bundle
\[
\bigl( (G \times \Delta_{\g}) \boxtimes (\SSS_N\otimes L^{2\nu}) \bigr)/K \to G\times_K N.
\]
Because $K$ acts trivially on the component $\Delta_{\kk}$ of $\Delta_{\g} = \Delta_{\kk} \otimes \Delta_{\p}$, the latter bundle equals
\[
\SSS_M \otimes L^{2\omega} \otimes \Delta_{\kk}\to M.
\]
(See Lemma 6.1 in \cite{HochsDS}.)

For $t\in [0,1]$, consider the operator
\[
\D_{G, t} := \sum_{j=1}^{d} X_j \otimes c(X_j) + t  \sum_{j=d+ 1}^{d_G} X_j \otimes c(X_j).
\]
 The point is that while this operator is only elliptic for $t>0$, the induced operator
\[
\bigl(\D_{G, t} \otimes 1 + 1\otimes\D_N^{L^{2\nu}}\bigr)^K
\]
on \eqref{eq invar sections} is also elliptic for $t = 0$. Indeed, by Proposition 4.7 in \cite{HochsDS}, it equals
\[
\bigl(\D_{G, 0} \otimes 1 + 1\otimes\D_N^{L^{2\nu}}\bigr)^K = \D_M^{L^{2\omega}} \otimes 1_{\Delta_{\kk}}.
\]

Therefore, varying $t$ from $1$ to $0$ yields an operator homotopy which implies that
\begin{equation} \label{eq KInd product}
\KInd_K^G\bigl[\D_N^{L^{2\nu}}\bigr] = \bigl[\D_M^{L^{2\omega}}\bigr] \otimes [\Delta_{\kk}] \quad \in K_{d_G}^G(M).
\end{equation}
Here $[\Delta_{\kk}] \in K_{d_K}^G(M)$ is the class of the zero operator on the trivial $G$- and $C_0(M)$-space $\Delta_{\kk}$. This space
 carries a natural $\Z_2$-grading if and only if $d_K$ is even, so it defines a class in the $d_K$'th $K$-homology group. Since 
 $\bigl[\D_M^{L^{2\omega}}\bigr] \in K_d^G(M)$, the product \eqref{eq KInd product} indeed lies in the $d_G$'th equivariant $K$-homology group of $M$.
We conclude that
\begin{equation} \label{eq quant ind 1}
\mu_M^G \circ \KInd_K^G \bigl[\D_N^{L^{2\nu}}\bigr] = Q_G(M)\otimes [\Delta_{\kk}]  \quad \in K_{d_G}(C^*_r(G)).
\end{equation}

The arguments in Subsection 5.4 of \cite{HochsDS} now imply that the right hand vertical maps in Diagram \eqref{diagram} map a class $[V] \in R(K)$ to
\begin{equation} \label{eq DInd}
\DInd_K^G[V] \otimes [\Delta_{\kk}] \quad \in K_{d_G}(C^*_r(G)).
\end{equation}
By  \eqref{eq quant ind 1} and \eqref{eq DInd}, commutativity of Diagram \eqref{diagram} implies that
\begin{equation} \label{eq quant ind 2}
Q_G(M)\otimes [\Delta_{\kk}] = \DInd_K^G \bigl(Q_K(N)\bigr) \otimes [\Delta_{\kk}]  \quad \in K_{d_G}(C^*_r(G)).
\end{equation}
Note that tensoring by the class $[\Delta_{\kk}]$ is the same as multiplying by the dimension of $\Delta_{\kk}$. 
So
since the $K$-theory of $C^*_r(G)$ has no torsion, \eqref{eq quant ind 2} reduces to
\[
Q_G(M) = \DInd_K^G \bigl(Q_K(N)\bigr) \quad \in K_d(C^*_r(G)),
\]
as required.
\end{proof}


\section{Proofs of the results}

We now specialise to the case where $G$ is a semisimple group, and later a complex semisimple group, or semisimple with discrete series. Proofs of the results in Subsections \ref{sec results} and \ref{sec special cases} are given in Subsection \ref{sec proof thm}, while the orbit methods in Subsection \ref{sec orbit results} are proved in Subsection \ref{sec orbit}.


\subsection{Quantisation commutes with reduction} \label{sec proof thm}

The proof of Theorem \ref{thm [Q,R]=0} is based on  the quantisation commutes with induction principle, 
Theorem \ref{thm quant ind}, and the fact that, in the symplectic setting, $\Spin^c$-quantisation commutes with reduction in the compact case (see Theorem 1.1 from \cite{Paradan}). This result states\footnote{This is the statement of the result in \cite{Paradan} if the image of the momentum map has nonempty intersection with the interior of the positive Weyl chamber. Otherwise, a more subtle $\rho$-shift should be used.} that the following diagram commutes for every dominant weight $\lambda  \in \Ldom$, with $\xi:= (\lambda+\rho_c)/i$:
\begin{equation} \label{eq [Q,R]=0 cpt}
\xymatrix{
\CHPS(K) \ar[r]^-{Q_K} \ar[d]_{\MWR_{\xi}} & R(K) \ar[d]^{R_K^{\lambda}}\\
\CHPS(\{e\}) \ar[r]^-{Q} & \Z. 
}
\end{equation}
Here $\MWR$ denotes Marsden-Weinstein reduction, including prequantisations and $\Spin^c$-structures. Note that we use only commutativity of this diagram for \emph{Hamiltonian} actions on symplectic manifolds, since it is not known in general if quantisation commutes with reduction for \emph{pre-Hamiltonian} actions on presymplectic manifolds.

Combining Diagram \eqref{eq RG RK} and 
Theorem \ref{thm quant ind} with Diagram \eqref{eq [Q,R]=0 cpt}, we obtain the following diagram:
\begin{equation} \label{eq diag proof}
\xymatrix{
\CpHPS(G) \ar[r]^-{Q_G} & K_{d}(C^*_r(G)) 
\ar@(dr, ur)[dd]^{R_G^{\lambda}}\\
\CHPS(K) \ar[r]^-{Q_K} \ar[d]_{\MWR_{\xi}}  \ar[u]^{\pHInd_K^G} & R(K) \ar[d]_{R_K^{\lambda}} \ar[u]^{\DInd_K^G}\\
\CHPS(\{e\}) \ar[r]^-{Q} & \Z. 
}
\end{equation}
Here we have applied the special case of Theorem \ref{thm quant ind} where the set $\ECpHPS(K)$ 
 is replaced by the subset $\CHPS(K)$.

Now let  $(N, \nu)$ be a compact, $\Spin^c$-prequantisable Hamiltonian $K$ manifold, and let $(M = G\times_K N, \omega)$ be the induced pre-Hamiltonian $G$-manifold  as in the definition of pre-Hamiltonian induction. By the Connes--Kasparov conjecture, there are integers $m_{\lambda}$, such that
\[
Q_G(M, \omega) = \sum_{\lambda  \in \Ldom} m_{\lambda} \clambda \quad \in K_{d}(C^*_r(G)).
\]
Commutativity of Diagram \eqref{eq diag proof} implies that, for all $\lambda  \in \Ldom$,
\[
m_{\lambda} = R_G^{\lambda} \bigl(Q_G(M, \omega)\bigr) 
	 = Q(N_{\xi}, \nu_{\xi}).
\]

Theorem \ref{thm [Q,R]=0} therefore follows from the following fact.
\begin{lemma} \label{lem sympl reds}
The map $j: N \to M$, given by 
\[
j(n)=[e, n],
\]
for $n \in N$,
induces a (pre)symplectomorphism between the symplectic reduction of $(N, \nu)$ by $K$ and the presymplectic reduction of $(M, \omega)$ by $G$, at $\xi$:
\[
j_G: (N_{\xi}, \nu_{\xi}) \xrightarrow{\cong} (M_{\xi}, \omega_{\xi}).
\]
In particular, the presymplectic reductions of $(M, \omega)$ are in fact symplectic. 
\end{lemma}
\begin{proof}
Consider the diagram
\[
\xymatrix{
(N, \nu) \ar[r]^-{j} & (M, \omega) \\\
\bigl(\Phi_N^{-1}(K\cdot\xi), \nu|_{\Phi_N^{-1}(K\cdot\xi)}\bigr) \ar[u]^-{\iota_N} \ar[r]^-j \ar[d]_-{q_N}& \bigl(\Phi_M^{-1}(G\cdot\xi), \omega|_{\Phi_M^{-1}(G\cdot\xi)} \bigr) \ar[u]^-{\iota_M} \ar[d]_-{q_M}\\
\bigl(N_{\xi} = \Phi_N^{-1}(K\cdot\xi)/K, \nu_{\xi}\bigr) \ar[r]^-{j_G}& \bigl(M_{\xi} = \Phi_M^{-1}(G\cdot\xi)/G, \omega_{\xi}\bigr).
}
\]
Here $\iota_N$ and $\iota_M$ are the respective inclusion maps, and $q_N$ and $q_M$ denote quotient maps. One can check that the induced map $j_G$ is a diffeomorphism, and  the maps $j$ are symplectic: $j^*\omega = \nu$. By some diagram chasing, the fact that $q_M^*\omega_{\xi} = \iota_M^*\omega$  implies that 
\[
q_N^* (j_G^*\omega_{\xi}) = \iota_N^*\nu,
\]
so that $j_G^*\omega_{\xi} = \nu_{\xi}$, as required. 
\end{proof}

For $G$ complex semisimple, it follows from \eqref{eq lambda pilamrho}  that Theorem \ref{thm [Q,R]=0} implies Corollary \ref{cor [Q,R]=0 ps}.

Finally, suppose $G$ is a real semisimple Lie group with discrete series. Let $\pi^d_{\mu}$ be an irreducible discrete series representation of $G$, where $\mu$ is the Harish--Chandra parameter of $\pi^d_{\mu}$ such that $(\alpha, \mu) \geq 0$ for all compact positive roots $\alpha$.  Let $\pidmu \in K_0(C^*_r(G))$ be the associated generator. The comment below Lemma 2.2.1 in \cite{Lafforgue} implies that this generator is of the form
\[
\pidmu = \pm\clambda,
\]
for a $\lambda \in \Ldom$. Lemma 1.5 in \cite{HochsDS} then yields the more explicit expression
\begin{equation} \label{eq class ds}
\pidmu = (-1)^{d/2}c(\mu - \rho_c).
\end{equation}
Therefore, Theorem \ref{thm [Q,R]=0} indeed implies Corollary \ref{cor [Q,R]=0 ds}.


\subsection{Orbits and generators} \label{sec orbit}

Fix an element $\lambda \in \Ldom$, and write $\xi := (\lambda + \rho_c)/i$ as before. Consider the manifold
\[
\widetilde{M}^{\lambda} := G\times_{K}(K\cdot \xi),
\]
constructed from $N = K\cdot \xi$ as in the pre-Hamiltonian induction procedure.  Let $\omega$ be the closed two-form on $M^{\lambda}$ induced by the Kostant--Kirillov symplectic form $\nu$ on $K \cdot \xi$ as in \eqref{eq ind form}.  The proof of Proposition \ref{prop orbit} is based on the following fact.
\begin{lemma} \label{lem M lambda}
There is a $G$-equivariant symplectomorphism
\[
(M^{\lambda}, p^*\omega^{\lambda}) \cong (\widetilde{M}^{\lambda}, \omega).
\]
\end{lemma}
\begin{proof}
The map
$
\widetilde{M}^{\lambda} \to M^{\lambda},
$
given by
\[
[g, k\cdot \xi] \mapsto gk K_{\xi}
\]
is a $G$-equivariant diffeomorphism. Under this identification, the map $p$ corresponds to the map\footnote{Note that the map $\tilde p$ (composed with the inclusion map $\O^{\lambda} \hookrightarrow \g^*$) is in fact the momentum map for the action  by $G$ on $\widetilde{M}^{\lambda}$ induced by the momentum map for the action by $K$ on $N$.}
$
\tilde p: \widetilde{M}^{\lambda} \to \O^{\lambda},
$
given by $\tilde p[g, k\cdot \xi] = gk\xi$.
It remains to show that $\tilde p^*\omega^{\lambda} = \omega$.

Let $X, X' \in \kk$ and $Y, Y' \in \p$ be given. For $k \in K$, a tangent vector in $T_{[e, k\xi]}M^{\lambda}$ has the form $Y + (X+ \kk_{\xi})$. We compute
\[
\begin{split}
\omega_{[e, k\xi]}\bigl(Y +  (X+ \kk_{\xi}), Y' +  (X'+\kk_{\xi})\bigr) &= \nu_{k\xi}(X+\kk_{\xi}, X' + \kk_{\xi}) - \langle k\cdot \xi , [Y, Y']\rangle \\
	&= \langle k\xi, [X, X'] + [Y, Y']\rangle \\
	& = \langle k \xi, [X + Y, X' + Y'] \rangle.
\end{split}
\]
Here we have used the facts that $[X, Y']$ and $[Y, X']$ are in $\p$, and $k \xi \in \kk^*$ annihilates $\p$.

On the other hand, one has
\[
T_{[e, k\xi]} \tilde p \bigl(Y +  (X+ \kk_{\xi}) \bigr) = X + Y + \kk_{\xi}
\]
(and similarly for $X'$ and $Y'$), so that
\[
\tilde p^*\omega^{\lambda}_{[e, k\xi]}\bigl(Y +  (X+ \kk_{\xi}), Y' + (X'+\kk_{\xi})\bigr) = \langle k\xi, [X+Y, X'+Y']\rangle.
\]
Therefore, the forms $\omega$ and $\tilde p^*\omega^{\lambda}$ are equal at points of the form $[e, k\xi]$, and hence on all of $M$ by $G$-invariance.
\end{proof}

Now we use the fact that
\[
Q_K(K\cdot \xi) = [V_{\lambda}] \quad \in R(K),
\]
the class of the irreducible representation $V_{\lambda}$ with highest weight $\lambda$. This is a consequence of the Borel--Weil(--Bott) theorem rather than of a general quantisation commutes with reduction theorem, since the momentum map of the action by $K$ on $K\cdot \xi$ has no regular values. For that reason, we no not obtain Proposition \ref{prop orbit} as a direct corollary to Theorem \ref{thm [Q,R]=0}, but Theorem \ref{thm quant ind} and  Lemma \ref{lem M lambda}  imply that
\[
Q_G(M^{\lambda}, p^*\omega^{\lambda}) = Q_G(\widetilde{M}^{\lambda}, \omega) 
 = \DInd_K^G \bigl(Q_K(K\cdot \xi) \bigr) 
 =  \DInd_K^G[V_{\lambda}] 
 = \clambda,
\]
as claimed.

Again, for complex semisimple groups $G$, the relation \eqref{eq lambda pilamrho} means that Proposition \ref{prop orbit} implies Corollary \ref{cor orbit ps} in that case.

In the setting of Corollary \ref{cor orbit ds}, it follows from \eqref{eq class ds} and Proposition \ref{prop orbit} that
\[
c(\mu-\rho_c) = Q_G(M^{\mu - \rho_c}, p^*\omega^{\mu - \rho_c}).
\]
However, since $\mu - \rho_c$ lies inside the strongly elliptic set, one has $G_{\xi} = K_{\xi}$, so that 
\[
M^{\mu - \rho_c} = G/G_{\mu - \rho_c} = \O^{\mu - \rho_c},
\]
with $p$ the identity map. Hence Corollary \ref{cor orbit ds} is true as well.

\bibliographystyle{amsplain}

\end{document}